\documentclass[11pt]{article}

\usepackage{amsmath}
\usepackage{amsfonts}
\usepackage{amssymb}
\usepackage{amsthm}
\usepackage{enumitem}
\usepackage{mathtools}
\usepackage{setspace}
\usepackage{etoolbox}

\newtheoremstyle{mystyle}
{11pt}                          
{11pt}                          
{}                                      
{}                                      
{\bfseries}                     
{}                                      
{5.5pt}                         
{}                                      

\theoremstyle{mystyle}
\newtheorem{theorem}{Theorem}[section]
\newtheorem{definition}[theorem]{Definition}
\newtheorem{lemma}[theorem]{Lemma}
\newtheorem{proposition}[theorem]{Proposition}
\newtheorem{corollary}[theorem]{Corollary}
\newtheorem{example}[theorem]{Example}

\renewenvironment{proof}[1][Proof.]{\vspace{-16.5pt} \begin{trivlist}
        \item[\hskip \labelsep {\bfseries #1}]}{\qed \end{trivlist}}

\setitemize{noitemsep, topsep=5.5pt, parsep=5.5pt, partopsep=0pt}

\allowdisplaybreaks
\appto\normalsize{
        \abovedisplayskip=5.5pt plus 2pt minus 2pt
        \belowdisplayskip=5.5pt plus 2pt minus 2pt
        \abovedisplayshortskip=5.5pt plus 2pt minus 2pt
        \belowdisplayshortskip=5.5pt plus 2pt minus 2pt}
\appto\small{
        \abovedisplayskip=5.5pt plus 2pt minus 2pt
        \belowdisplayskip=5.5pt plus 2pt minus 2pt
        \abovedisplayshortskip=5.5pt plus 2pt minus 2pt
        \belowdisplayshortskip=5.5pt plus 2pt minus 2pt}

\newcommand{\gap}{\vspace{11pt}}
\newcommand{\sgap}{\vspace{5pt}}

\newcommand{\tr}{\operatorname{tr}}

\newcommand{\R}{\mathcal{R}}
\newcommand{\LL}{\mbox{\rm LL}}
\newcommand{\Rn}{\mathcal{R}^n}

\newcommand{\Sn}{\mathcal{S}^n}

\newcommand{\Ln}{{\mathcal L}^n}

\newcommand{\Hn}{\mathcal{H}^n}

\newcommand{\V}{{\cal V}}

\title{\bf Completely mixed linear games and irreducibility concepts for $Z$-transformations over  self-dual cones}
\author{
        M. Seetharama Gowda\\
        Department of Mathematics and Statistics\\
        University of Maryland, Baltimore County\\
        Baltimore, Maryland 21250, USA\\
        gowda@umbc.edu
}

\date{May 7, 2024}

\begin{document}

\maketitle

\begin{abstract}
In the setting of a self-dual cone in a finite-dimensional inner product space, we consider  (zero-sum) linear games. In our previous work \cite{gowda-value, gowda-ravindran}, we 
showed that a $Z$-transformation with positive value is completely mixed. The present paper considers the case when the value is zero. Motivated by the matrix game result that a $Z$-matrix with value zero is
completely mixed if and only if it is irreducible, we formulate our general results based on the
concepts of cone-irreducibility and space-irreducibility.  While the concept of cone-irreducibility for a positive linear transformation is well-known, we introduce space-irreducibility for a general linear transformation by reformulating the irreducibility concept of Elsner \cite{elsner}. Our main result is that for a $Z$-transformation with value zero, space-irreducibility is necessary and sufficient for the completely mixed property. 
We also extend a recent result of Parthasarathy et al. \cite{parthasarathy et al} on matrix games with value zero to the setting of a symmetric cone (in a Euclidean Jordan algebra). Additionally,
we present a refined cone/space-irreducibility result for positive transformations on symmetric cones. 
\end{abstract}

\vspace{1cm}
\\

\\

{\bf Dedication}: This paper is dedicated to the memory of Professor T. Parthasarathy who passed away on September 22, 2023.  A distinguished Indian mathematician, Professor Parthasarathy made substantial contributions to the areas of game theory and linear complementarity problems. 
\newpage

\section{Introduction}

In the classical (zero-sum) matrix game setting, Kaplansky's well-known result  (\cite{kaplansky-1}, Theorem 5) provides a necessary and sufficient condition for a matrix game with value zero to be completely mixed. About 50 years later, Kaplansky \cite{kaplansky-2} gave another such result for skew-symmetric matrices of odd order in terms of the so-called Pfaffians. Considering $Z$-matrices (which are square real matrices with nonpositive off-diagonal entries), Raghavan \cite{raghavan} showed that if the corresponding matrix game has positive value, then it is completely mixed and provided several characterizations of nonsingular $M$-matrices. 
In \cite{gowda-value, gowda-ravindran},  the concepts of value of a  matrix game and completely mixed property were
 generalized to the setting of a linear transformation over a self-dual cone in a finite dimensional real inner product space;  for related work on symmetric games and proper cones, see \cite{gokulraj-chandrashekaran-1, gokulraj-chandrashekaran-2, orlitzky}.
 In \cite{gowda-value, gowda-ravindran} , the results of Kaplansky \cite{kaplansky-1} and Raghavan \cite{raghavan} were generalized and, in particular, it was shown that for a $Z$-transformation (which is a generalization of a $Z$-matrix), positive value implies the completely mixed property. In the present paper, we focus on $Z$-transformations with value zero and provide a necessary and sufficient condition for the completely mixed property in terms of an irreducibility condition.
\\
To elaborate, we start with a brief description of relevant game-theoretic ideas.
Let $(\V, \langle \cdot,\cdot\rangle)$  be a finite dimensional real inner product
space. Consider a {\it self-dual cone} $K$ in $\V$, so $K=K^*:=\{x\in \V: \langle x,y\rangle \geq 0,\,\,\forall\,y\in K\}.$ We fix an element $e$ in $K^\circ$ (the interior
of $K$) and let
$$\Delta(e) := \{x\in K : \langle x, e\rangle = 1\},$$
the elements of which will be called ‘strategies’. Given a linear transformation $L$ on $\V$, the zero-sum linear game — denoted by $(L, e)$ — is played by two players I and
II in the following way: If player I chooses strategy $x\in  \Delta(e)$ and player II chooses
strategy $y\in \Delta(e)$, then the pay-off for player I is $\langle L(x), y\rangle$ and the pay-off for player II
is $-\langle L(x), y\rangle$. Both players try to maximize their pay-offs. Since $\Delta(e)$ is a compact convex
set and $L$ is linear, by the min–max Theorem of von Neumann (\cite{karlin}, Theorems 1.5.1
and 1.3.1), there exist  strategies $\overline{x}$ (for player I) and $\overline{y}$ (for player II) satisfying the inequalities
$$\langle L(x), \overline{y}\rangle\leq \langle L(\overline{x}), \overline{y}\rangle\leq \langle L(\overline{x}), y\rangle\quad\mbox{for all}\,\, x, y\in \Delta(e).$$
 (So the players I and II do not gain by unilaterally changing their strategies
from $\overline{x}$ and $\overline{y}$.) The number
$$v(L, e) := \langle L(\overline{x}), \overline{y}\rangle$$
is called the {\it value of the game} and  the pair $(\overline{x},\overline{y})$ is called an {\it optimal
strategy pair}. This value is also given by (\cite{karlin}, Theorems 1.5.1 and 1.3.1)
$$v(L, e) = \max_{ x\in \Delta(e)}\min_{y\in \Delta(e)}\langle L(x), y\rangle = \min_{y \in \Delta(e)} \max_{x\in\Delta(e)}\langle L(x), y\rangle.$$
{\it We say that the game $(L,e)$ is completely mixed if for every optimal strategy pair $(\overline{x},\overline{y})$, both $\overline{x}$ and $\overline{y}$ belong to $ K^\circ$. (As we see below, in this case, there is a unique optimal strategy pair.)}\\

We note that with $\V=\Rn$ (under the usual inner product), $K=\Rn_+$ (nonnegative orthant), and $e$ denoting the vector of ones, we get back to the classical {\it matrix-game setting}. 
 \\

  For ease of reference, we recall some results from \cite{gowda-value, gowda-ravindran}; here we use the notation  $x\geq 0$ $(x>0, \,x\leq 0,\,x\leq y)$ to mean $x\in K$ (respectively, $x\in K^\circ,\, -x\geq 0,\,x-y\leq 0)$ and write $L^T$ for the transpose of $L$.\\

\begin{theorem}\label{basic results} \cite{gowda-value, gowda-ravindran}
{\it  For the linear game $(L,e)$, the following statements hold:
\begin{itemize}
\item [$(1)$]  If $(x,y)$ is an optimal strategy pair, then $L^T (y)\leq  v(L,e)\, e\leq  L(x)$. Conversely, if there exist  $x,y\in \Delta(e)$ and a real number $v$ such that $L^T (y)\leq  v\, e\leq  L(x)$,
then $v = v(L, e)$ and $(x, y)$ is an optimal strategy pair for $(L, e)$; additionally, if
there exists an optimal strategy pair $(\overline{x},\overline{y})$ with $\overline{x},\overline{y}>0$, then $L^T (y)=  v(L,e)\, e=  L(x)$.
\item [$(2)$] If $(x, y)$ is an optimal strategy pair for $(L, e)$, then $(y, x)$ is an optimal strategy pair for $(-L^T, e)$; moreover, $v(-L^T,e)=-v(L,e)$.
\item [$(3)$] $v(L,e)>0$ if and only if there exists $u>0$ such that $L(u)>0$.
\item [$(4)$] The sign of $v(L,e^\prime)$ is a constant as $e^\prime$ varies over $K^\circ$.
\item [$(5)$]  If for every optimal strategy pair $(x,y)$ of $(L, e)$, $y>0$ $($equivalently, $x>0$ for every optimal strategy pair $(x,y))$, then $(L,e)$  is completely mixed; in this case,  there is a unique  optimal strategy pair.
\item [$(6)$] When $v(L, e)= 0$, $(L,e)$ is completely mixed if and only if $dim(ker(L))=1$ and there exist  $\overline{x},\,\overline{y}>0$ such that for the transformation S := $\overline{x}\, \overline{y}^T$
$($which takes any $x$ to $\langle \overline{y},x\rangle\,\overline{x})$, we have $LS = SL = 0$. 
\item [$(7)$] When $(L, e)$ is completely mixed, $v(L, e)\neq  0$ if and only if $L$ is invertible.
\item [$(8)$] If $(L,e)$ has value zero and is completely mixed, then for every $e^\prime\in K^\circ$, the game $(L,e^\prime)$ has value zero and is completely mixed.

\end{itemize}
}
\end{theorem}

Motivated by matrix theory and dynamical systems considerations, some special types of linear transformations were studied in \cite{gowda-value, gowda-ravindran}.
We recall that a  linear transformation $L$ on $\V$ (relative to a proper cone $K$  with dual $K^*$) is said to be  a
{\it 
\begin{itemize}
\item [$(a)$]  $Z$-transformation  if $\big [x\in K, y\in K^*, \langle x, y\rangle = 0\big ]\Rightarrow  \langle L(x), y\rangle\leq  0$,
\item [$(b)$] Lyapunov-like transformation if $\big [x\in K, y\in K^*, \langle x, y\rangle = 0\big ]\Rightarrow  \langle L(x), y\rangle= 0$, 
\item [$(c)$]  Stein-like transformation if $L = I-S$, where $S\in \overline{Aut(K)}$ with $Aut(K)$ denoting the set of all bijective linear transformations from $K$ to $K$, and
\item [$(d)$] positive transformation if $L(K)\subseteq K$.
\end{itemize}
}
 
 We note that in the setting of $\V=\R^n$ with $K=\R^n_+$, a $Z$-transformation is just a $Z$-matrix. Also, in the setting of 
$\V=\Sn$ (the space of all $n\times n$ real symmetric matrices) with $K=\Sn_+$ (the cone of all positive semidefinite matrices in $\Sn$), corresponding to $A\in \R^{n\times n}$, 
the transformations defined by $L_A(X):=AX+XA^T$ and $S_A(X):=X-AXA^T$ are, respectively, examples of Lyapunov-like and Stein-like transformations. As is well-known, these two transformations, respectively called Lyapunov and Stein transformations, appear in the study of continuous and discrete dynamical systems. 

\gap

We recall the following results from \cite{gowda-ravindran}; note that in these results, $K$ is assumed to be self-dual.

\begin{theorem} \label{GR- nonzero value implies cm}
{\it
The game $(L,e)$ is completely mixed in each of the following instances:
 \begin{itemize}
\item [$(1)$]  $L$ is a $Z$-transformation with $v(L,e)>0$.
\item [$(2)$]  $L$ is Lyapunov-like with $v(L,e)\neq 0$.
\item [$(3)$]  $L$ is Stein-like with $v(L,e)\neq 0$.
\end{itemize}
}
\end{theorem}

What happens if the value is zero? Specifically, we ask:
\begin{center}
{\it For a $Z$-transformation with value zero, when is the game completely mixed?}
\end{center}
In the matrix game setting, we have the result (see Theorem \ref{z-matrix result} below) that a {\it $Z$-matrix with value zero is completely mixed if and only if it is irreducible.}
Motivated by this, we formulate our results based on the irreducibility concept(s). 
In the setting of a proper cone $K$, there is the known concept of $K$-irreducibility (to be called cone-irreducibility in this paper) for a positive 
 transformation \cite{berman-plemmons}: The only faces of $K$ that are invariant under the transformation are $\{0\}$ and $K$. 
For a general linear transformation over a finite dimensional real inner product space (still relative to a proper cone), we define space-irreducibility 
 by stipulating that the transformation leaves no non-trivial subspace of the form $F-F$ (where $F$ is a face of $K$)   invariant, see Definition \ref{irreducibility defn over a proper cone} below. 
 It turns out that this space-irreducibility concept is equivalent to the irreducibility concept introduced by Elsner \cite{elsner}. 
Then, based on the exponential characterization of a cross-positive transformation (which is the negative of a $Z$-transformation) of Schneider and Vidyasagar \cite{schneider-vidyasagar} and the  irreducibility result of Elsner \cite{elsner}, we prove our main result (in Theorem \ref{space-irreducibility equals cm}):

 \begin{itemize}
 \item {\it  Over a self-dual cone,  space-irreducibility is necessary and sufficient for a $Z$-transformation with value zero to be completely mixed.}
 \end{itemize}

By way of an application, see Corollary \ref{generalization of encinas et al result}, we show that a space-irreducible $Z$-transformation with value zero has `almost monotone', `trivially range monotone,'  and group inverse properties; these latter properties were studied recently by Encinas, Mondal, and Sivakumar \cite{encinas et al} for Lyapunov and Stein
 transformations over the space of all  $n\times n$ real symmetric matrices.

\sgap

Next, in Section 4 we describe several results and examples in the setting of symmetric cones (which are self-dual and homogeneous) in  Euclidean Jordan algebras. In particular, we 
\begin{itemize}
{\it \item  generalize the following result of 
 Parthasarathy, Ravindran, and Sunil Kumar  \cite{parthasarathy et al} on matrix games:
{\it For a real $n\times n$ matrix $A$ with $v(A)=0$,  $A$ is completely mixed if and only if $v(A+D_i)>0$ for all $i=1,2,\ldots, n$, where $D_i$ denotes the diagonal matrix with $1$ in the $(i,i)$ slot and zeros elsewhere},
\item provide a finer characterization of cone-irreducibility of a positive transformation $T$: For each $0\neq x\geq 0$,  $(I+T)^{n-1}x>0$, where $n$ denotes the rank of $\V$, 
\item characterize the completely mixed property of a symmetric/skew-symmetric Lyapunov-like transformation, and
\item provide a necessary and sufficient condition for the space-irreducibility of a Lyapunov transformation $L_A$ on $\Sn$.
}
\end{itemize}

While the focus of the paper is on  game-theoretic ideas/results, one may consider studying $Z$-transformations with value zero for their dynamical systems and complementarity properties.
In the case of a $Z$-transformation with value positive (equivalently, positive stable, see Theorem \ref{positive stability of z}),  the {\it  dynamical system}
$$\frac{dx}{dt}+L(x)=0$$
is (globally) asymptotically stable, that is, starting from any initial point, the trajectory of the above system converges to zero as (time) $t\rightarrow \infty$. The question of what happens when the value is zero is left for a future study.
\\
Given a linear transformation $L$, a proper cone $K$, and a $q\in \V$, the {\it linear complementarity problem} LCP$(L,K,q)$ is to find $x\in \V$ such that 
$$x\in K,\,\,L(x)+q\in K^*,\,\mbox{and}\,\,\langle x,y\rangle =0.$$
When $L$ is a $Z$-transformation (on a self-dual cone $K$)  with value positive, this problem has a solution for all $q\in \V$ \cite{gowda-tao-z}; in the setting of $\V=\Rn$ and $K=\Rn_+$, this can be described in more than 50 equivalent ways, see \cite{berman-plemmons}. What happens to such properties when the value is zero? Again, we relegate this to a future study.

\section{Preliminaries}
Throughout this paper, {\it $\V$ denotes a finite dimensional real inner product space and $K$ is a proper cone in $\V$ (that is, $K$ is a pointed closed convex cone with nonempty interior).  In all game-theoretic settings, we assume that $K$ is self-dual. In particular, when $\V$ is a Euclidean Jordan algebra, we assume that $K$ is the corresponding symmetric cone.}  The interior, closure,  boundary, and span of a set $X$ in $\V$ are denoted, respectively, by  $X^\circ$, $\overline{X}$,  $\partial(X)$, and $\rm{span}(X)$. For a linear transformation $L$, $ker(L)$ denotes its kernel (null space). Recall that a convex cone $F$ within (the proper cone) $K$ is a {\it  face} if an element $a$ in $F$ is of the form $b+c$  with $b,c\in K$, then $b,c\in F$. Since $K$ is proper, every face of $K$ is a pointed closed convex cone in $\V$.  If $F$ is a face of $K$, then the $span(F): =F-F$ is the subspace generated by $F$. For various definitions and properties regarding cones, proper cones, faces, etc., we refer to \cite{berman-plemmons}.\\
Throughout, we use the following notation: Relative to the given proper cone $K$ in $\V$, we write
$$x\geq 0 \,\,\mbox{when}\,\,x\in K\,\,\mbox{and}\,\,x>0\,\,\mbox{when}\,\,x\in K^\circ\,(=\mbox{interior of}\,K).$$
With  $K^*$ denoting the dual of $K$, we  frequently use the well-known fact
\begin{equation}\label{well-known fact about dual}
\big [ x\in K^\circ, 0\neq y\in K^*\big ] \Rightarrow \langle y,x\rangle >0.
\end{equation}

\gap

We freely use the game-theoretic concepts/results formulated in the Introduction. Throughout this paper, we will be focusing on describing the completely mixed property for linear transformations with value zero
relative to some $e\in K^\circ$. In view of Item $(8)$ in Theorem \ref{basic results}, the element $e$ can be replaced by any other $e^\prime\in K^\circ$. So, 
\begin{center}
{\it when $v(L,e)=0$ and $(L,e)$ is completely mixed, we suppress $e$ and \\write $v(L)=0$ and say that $L$ is completely mixed.}
\end{center}
\subsection{Z-transformations}
For the given proper cone $K$ in $\V$, we consider the following sets of linear transformations:
$$
{\it  \begin{array}{l}
 Z(K):=\mbox{Set of all Z-transformations on}\, K,\\
\LL(K):=\mbox{Set of all Lyapunov-like transformations on}\, K,\\ 
\pi(K):=\mbox{Set of all positive transformations on K}.
\end{array}
}
$$

When $K$ is the nonnegative orthant in $\Rn$, these become, respectively, the set of $Z$-matrices, diagonal matrices, and nonnegative matrices. 

\gap

We recall the following results of Schneider and
Vidyasagar \cite{schneider-vidyasagar}. In this reference, the authors deal with a `cross-positive matrix'  which is the negative of a $Z$-transformation.

\begin{theorem}  (\cite{schneider-vidyasagar}, Theorems  3 and 6)\label{schneider-vidyasagar}
{\it 
\begin{equation}\label{schneider-vidyasagar result}
L\in Z(K)\Leftrightarrow exp(-tL)\in \pi(K)\,\,\mbox{for all}\,\, t\geq  0,
\end{equation}
where  $exp(L):=\sum_{k=0}^{\infty}\frac{L^k}{k!}$.
Moreover, When $L\in Z(K)$, 
$$\alpha:=\min\,\{Re(\lambda): \lambda\,\,\mbox{is an eigenvalue of}\,\,L\}$$ 
is an eigenvalue of $L$ with an associated eigenvector $u\in K$. 
}
\end{theorem}

We mention one interesting consequence:
\begin{proposition}
{\it Suppose $K$ is self-dual, $e\in K^\circ$, and  $L\in Z(K)$. If $(L,e)$ is completely mixed, then $v(L,e)$ and $\alpha$ have the same sign; in particular, when $v(L,e)$ is zero, we have $\alpha=0$ and so $L$ is nonnegative stable.}
\end{proposition}
\begin{proof}
Assume that $(L,e)$ is completely mixed.  Then, by Theorem \ref{basic results}, there exist $\overline{x},\overline{y}>0$ such that $L^T(\overline{y})=v(L,e)\,e=
L(\overline{x})$. From Theorem \ref{schneider-vidyasagar}, there exists (an eigenvector) $u$ such that $0\neq u\geq 0$ and  $L(u)=\alpha u$; without loss of generality, let $\langle e,u\rangle =1$. Then, $v(L,e)=\langle v(L,e)\,e,u\rangle=\langle L^T(\overline{y}),u\rangle =\langle \overline{y}, L(u)\rangle=\alpha \langle \overline{y},u\rangle$.
 As $\langle \overline{y},u\rangle >0$ (from (\ref{well-known fact about dual})),  we conclude that   $v(L,e)$ and $\alpha$ have the same sign. In particular, when the $v(L,e)$ is zero, $\alpha=0$; so,  the real part of any eigenvalue of $L$ is nonnegative, that is, $L$ is nonnegative stable.
\end{proof}

\gap

We also have the following from \cite{gowda-tao-z}, Theorem 6 and \cite{gowda-ravindran}, Theorem 6:
\begin{proposition} \label{positive stability of z}
 {\it  The following are equivalent for $L\in Z(K)$:
 \begin{itemize}
\item $L$  is positive stable (that is, the real part of any eigenvalue of $L$ is positive).
\item There exists $u>0$ such that $L(u)>0$. 
\item  $L^{-1}$ exists and belongs to $\pi(K)$.
\end{itemize}
When $K$ is self-dual and $e\in K^\circ$,  these are further equivalent to: $v(L,e)>0$.
Moreover, in this case, $(L,e)$ is completely mixed.}
\end{proposition}

\subsection{Irreducibility concepts for a linear transformation relative to a proper cone}
Here we describe two irreducibility concepts, both relative to a  proper cone $K$ (with dual $K^*$) in a finite dimensional real inner product space $\V$. The first one, which we call cone-irreducibility, is known. It is defined (only) for those in  $\pi(K)$. Recall that $L\in \pi(K)$ if $L$ is linear and  $L(K)\subseteq K$.
\\
The following definition and result(s) are from  \cite{berman-plemmons}, where they are stated for $\Rn$; by isomorphism considerations, we state them here for $\V$. 
\begin{definition} (\cite{berman-plemmons}, Definition 3.14)
{\it Let $K$ be a proper cone in $\V$. A transformation $T$ in $ \pi(K)$ is  said to be $K$-irreducible or cone-irreducible if  the only faces of $K$ that are invariant under $T$ are $\{0\}$ and $K$.}
\end{definition} 
 
\sgap

\begin{proposition} (\cite{vandergraft} and \cite{berman-plemmons}, Chapter 1, Section 3)\label{berman-plemmons result on K-irreducibility}
{\it Let $K$ be a proper cone in $\V$ and  $T\in \pi(K)$. Then the following are equivalent:
\begin{itemize}
\item [$(a)$] $T$ is $K$-irreducible.
\item [$(b)$] $T$ has no eigenvectors on  the boundary of $K$.
\item [$(c)$] Up to scalar multiples, $T$ has exactly one  eigenvector in $K$, and this eigenvector belongs to $K^\circ$.
\item [$(d)$] The transpose of $T$ is $K^*$-irreducible.
\item [$(e)$] $(I+T)^{n-1}\big ( K\backslash \{0\}\big )\subseteq K^\circ,$ where $n$ is the dimension of $\V$.
\end{itemize}
}
\end{proposition}

Our second irreducibility concept is motivated by the usual irreducibility concept of a matrix.
Recall that a real square matrix is said to be irreducible if no (simultaneous) permutation of rows and columns produces a $2\times 2$ block matrix in which the diagonal blocks are square and the southwest (equivalently, northeast) block is zero.
 We extend this to a general linear transformation (but still based on the proper cone $K$):

\begin{definition}\label{irreducibility defn over a proper cone}
{\it Let $K$ be a proper cone in $\V$ and $L$ be linear on $\V$.  Then $L$ is said to be  space-irreducible if 
 $$\Big [ F\,\,\mbox{is a face of}\,\, K\,\mbox{with}\,\,L(F-F)\subseteq F-F\Big ] \Rightarrow F=\{0\}\,\mbox{or}\, K.$$}
\end{definition}

We say that $L$ is space-reducible if it is not space-irreducible. Trivially, when $\V$ has dimension one, every linear transformation is space-irreducible. A few simple observations are recorded below.

\begin{proposition}\label{simple observations}
{\it Let $K$ be a proper cone in $\V$ and $L$ be linear on $\V$. Then, the following hold:
\begin{itemize}
\item [$(a)$] If $L$ is space-irreducible, then $\{0\}$ and possibly $K$ are the only faces of $K$ that are invariant under $L$.
\item [$(b)$] When $L\in \pi(K)$,  space-irreducibility is the same as cone-irreducibility.
\item [$(c)$] When    $L=rI-S$, where $r\in \R$ and $S\in \pi(K)$,  space-irreducibility of $L$ is equivalent to the cone-irreducibility  of $S$.
\item  [$(d)$] When $\V=\Rn$ and $K=\Rn_+$, space-irreducibility reduces to the usual/classical one.
\end{itemize}}
\end{proposition}

\begin{proof}
$(a)$ Suppose  $L$ is space-irreducible and let $F$ be  a face of $K$ with $L(F)\subseteq F.$ Then, $L(F-F)=L(F)-L(F)\subseteq F-F.$
By space-irreducibility, $F=\{0\}$ or $K$.\\
$(b)$ Suppose $L\in \pi(K)$. If $L$ is space-irreducible, then by $(a)$,  $L$ is cone-irreducible. Now assume that $L$ is cone-irreducible and let  $L(F-F)\subseteq F-F$, where $F$ is a face of $K$. Consider any $a\in F$. Then, $L(a)$, which is in $L(F-F)$ is of the form $b-c$, where $b,c\in F$. So, $L(a)=b-c$. Since $L\in \pi(K)$ and $a\in F\subseteq K$, we have $L(a)\in K$. As $b=c+L(a)$ and $F$ is a face of $K$, we see that $L(a)\in F$.  
Hence, $L(F)\subseteq F$. From the cone-irreducibility, $F=\{0\}$ or $K$. Thus, $L$ is space-irreducible.\\
$(c)$ Clearly, $L=rI-S$ is space-irreducible if and only if $S$ is space-irreducible;  the stated assertion follows from $(b)$.
\\
$(d)$ This statement is easy to see as every face of $\Rn_+$, upto a permutation, is of the form $\R^k_+\times \{0\}$.
\end{proof}

\gap

In \cite{elsner}, Elsner introduced the following type of irreducibility (which we shall call $E$-irreducibility).
Let $K$ be a proper cone in $\V$. Recall the  order relative to $K$: $x\leq y$ if $y-x\in K.$ For any $y\in K$, let 
$$[y]:=\{x\in \V: 0\leq x\leq y\}\quad\mbox{and}\quad S_y:=\mathrm{span}( [y]).$$

\begin{definition}\cite{elsner}\label{elsner's defn}
{\it Let $K$ be a proper cone in $\V$ and $L$ be linear on $\V$.  Then $L$ is said to be E-irreducible  if for any $y\in K$,
 $$ L(S_y)\subseteq S_y\, \Rightarrow S_y=\{0\}\,\,\mbox{or}\,\,S_y=\V.$$
 }
\end{definition}

We now have an important observation (due to Orlitzky \cite{orlitzky-seminar}): {\it  The space-irreducibility is the same as E-irreducibility}:\\This is seen as follows. For any $y\in K$,   $[y]=K\cap (y-K)$. Consider the face $F$ generated by $y$ (which is the smallest face containing $y$) so $y\in relint(F)$, that is, $y$ is in the relative interior of $F$.  Then, by the definition of face, $K\cap (y-K)=F\cap (y-F).$ As $y\in relint(F)$, 
$\mathrm{span}( [y])=\mathrm{span} ( F\cap (y-F))=F-F$.  Hence $S_y=F-F$. On the other hand, suppose $F$ is a face of $K$; take any $y$ that is in the relative interior of $F$.  Then $K\cap (y-K)=F\cap (y-F)$ and so $S_y=\mathrm{span}([y])=F-F$. Thus, for each $y\in K$, there is a face $F$ such that $S_y=F-F$ and  conversely, for every face $F$, there is some $y\in K$ such that $S_y=F-F$. It is now easy to see that the two definitions are equivalent.

\gap

Now, based on the above equivalence and the results in \cite{elsner} (where Elsner states his  results for quasimonotone/cross-positive matrices, i.e., for the negatives of our $Z$-transformations), we have the following:

\begin{theorem} \label{E and space-irreducibility}
{\it 
Let $K$ be a proper cone in $\V$ and $L\in Z(K)$. Then, relative to $K$, the following are equivalent:
\begin{itemize}
\item [$(i)$] $exp(-tL)$ is cone-irreducible for some $t>0$.
\item [$(ii)$] $L$ is space-irreducible (equivalently, $E$-irreducible).
\item [$(iii)$] $L$ has no eigenvector on the boundary of $K$.
\item [$(iv)$]  Except for a countable number of $t$s in the interval $(0,\infty)$, $exp(-tL)$ is cone-irreducible.
\end{itemize}
}
\end{theorem}

\begin{proof}
$(i)\Rightarrow (ii)$:  As $L\in Z(K)$, we know from (\ref{schneider-vidyasagar result}), $exp(-tL)\in \pi(K)$ for all $t\geq 0$. Consider/fix  a $t>0$ for which $exp(-tL)$ is $K$-irreducible. Suppose, if possible, 
$L(F-F)\subseteq F-F$ for some face $F$ of $K$. Let $W:=F-F$. Then, for every  $k\in \{0,1,2\ldots\}$, $(-tL)^k$ keeps the subspace $W$ invariant. Therefore $exp(-tL)$ keeps  $W$ invariant, i.e., $exp(-tL)(F-F)\subseteq F-F$.  . Now take any $x\in F$. Then, $y:= exp(-tL)x\in K$ and $exp(-tL)x=a-b$, where $a,b\in F$. Then, $a=y+b$ with $a\in F$, $y\in K$, and $b\in F\subseteq K$. Since $F$ is a face,
$y$ must be in $F$. This shows that for every $x\in F$, $exp(-tL)x\in F$. Thus, $exp(-tL)$ keeps $F$ invariant. By our assumption, $F=\{0\}$ or $K$. This proves  that $L$ is space-irreducible.
\\
$(ii)\Rightarrow (iii)$: See \cite{elsner}, Satz 1 and Satz 3.\\
$(iii)\Rightarrow (iv)$: See \cite{elsner},  Satz 3.\\
$(iv)\Rightarrow (i)$: Obvious.
\end{proof}

\noindent{\bf Remarks.}  It is easy to see that  $L\in Z(K)$ if and only if $L^T\in Z(K^*)$. Now suppose $L\in Z(K)$ and is space-irreducible relative to $K$. Then, by the above result, 
$exp(-tL)$ is $K$-irreducible for some $t>0$. By the equivalence $(a)\Leftrightarrow (d)$ in Proposition 
\ref{berman-plemmons result on K-irreducibility}, $exp(-tL^T)$ is $K^*$-irreducible. It follows that $L^T$ is space-irreducible relative to $K^*$. Since $(L^T)^T=L$ and $(K^*)^*=K$, we see that 
\begin{center}
{\it When  $K$ is a proper cone and $L\in Z(K)$, \\$L$ is space-irreducible relative to $K$ if and only if $L^T$ is space-irreducible relative to $K^*$.}
\end{center}

\sgap  

We note that  Item $(i)$ in the above theorem deals with a scaled $L$. In particular, the cone-irreducibility of $exp(-L)$ implies the space-irreducibility of $L$. We may ask when the converse holds.  Here is a partial answer.

\begin{corollary}
{\it Let $K$ be a proper cone in $\V$ and $L=rI-S$, where $r\in \R$ and $S\in \pi(K)$. Then, $L$ is space-irreducible if and only if $exp(-L)$ is cone-irreducible.
}
\end{corollary}

\begin{proof}
Clearly, $L\in Z(K)$.  If $exp(-L)$ is cone-irreducible, then by Theorem \ref{E and space-irreducibility}, $L$ is space-irreducible. Now suppose $L$ is space-irreducible and let $F$ be a face of $K$ such that  $exp(-L)(F)\subseteq F$ with $F\neq \{0\}, K$. By Proposition \ref{simple observations}, $S$ is cone-irreducible. As $rI$ and $S$ commute, from  $exp(-L)(F)\subseteq F$ we see that  $exp(S)(F)\subseteq F$.  Let $0\neq x\in F$ so that $exp(S)x\in F\subseteq\partial(K)$. By the supporting hyperplane theorem,
there exists $0\neq y\in K^*$ such that $\langle exp(S)x,y\rangle =0.$ Expanding $exp(S)$ as a series, we see that $\langle S^k(x),y\rangle=0$ for all $k=0,1,2\ldots$. Then, $\langle (I+S)^{n-1}x,y\rangle=0$, where $n$ is the dimension of $\V$. However, as $S$ is cone-irreducible, from Theorem \ref{berman-plemmons result on K-irreducibility}, $(I+S)^{n-1}x\in K^\circ$. This implies, because of (\ref{well-known fact about dual}),   $\langle (I+S)^{n-1}x,y\rangle> 0$. We reach a
 a contradiction. Thus, $F=\{0\}$ or $K$; so $exp(-L)$ is cone-irreducible.
\end{proof}

\section{The completely mixed property}
In this section, we present our main result describing the completely mixed property of a $Z$-transformation with value zero.
For motivation, we present the following result in the classical matrix game setting which, perhaps, is known. (See  Theorem 3.7 in \cite{li} for an implicit formulation dealing with a singular $Z$-matrix.)  For lack of a precise reference, we state the result with its proof. 

\begin{theorem}\label{z-matrix result}
{\it Suppose $A$ is a $Z$-matrix with value zero. Then $A$ is completely mixed if and only if it is irreducible.
}
\end{theorem}

\begin{proof}
Let $A\in \R^{n\times n}$. Suppose $A$ is irreducible. Consider a strategy $\overline{x}\in \R^n$ such that $A\overline{x}\geq 0$. Suppose $\overline{x}$ has a zero component. Modulo a permutation, we may assume that the first $k$ components of $\overline{x}$ are nonzero, where $1\leq k<n$. Let $\overline{x}_1$ be the vector in $\R^k$ with these nonzero components. Correspondingly,  we rewrite $A$ as a  $2\times 2$ block matrix with square diagonal blocks and $A_3$ in the southwest corner. Then, $A\overline{x}\geq 0$ implies that 
$A_3\overline{x}_1\geq 0$. Since $A_3\leq 0$ (as $A$ is a $Z$-matrix) and $\overline{x}_1>0$, we must have $A_3=0$. As $A$ is assumed to be irreducible, this cannot happen. Hence, $\overline{x}>0$. This proves that $A$ is completely mixed.\\
Now suppose $A$ is completely mixed. Let $(\overline{x},\overline{y})$ be the unique strategy pair such that $A^T\overline{y}=0=A\overline{x}$. We also know from Kaplansky's result \cite{kaplansky-1}  that the $ker(A)$ has dimension $1$.  Assume if possible, $A$ is reducible. Without loss of generality,  we write (in the block form)
$$A=\left [ \begin{array}{cc} A_1&A_2\\0 & A_4\end{array}\right ],\quad \overline{x}=
\left [ \begin{array}{c}
\overline{x}_1 \\ \overline{x}_2\end{array} \right ],\quad\mbox{and}\quad  \overline{y}=
\left [ \begin{array}{c}
\overline{y}_1 \\ \overline{y}_2\end{array} \right ].
$$
From $A^T\overline{y}=0=A\overline{x}$, we get $A_1\overline{x}_1+A_2\overline{x}_2=0=A_1^T\overline{y}_1$. Since $A$ is a $Z$-matrix, $A_2\leq 0$; hence, $A_1\overline{x}_1\geq 
-A_2\overline{x}_2\geq 0$. Thus, $A_1^T\overline{y}_1=0\leq A_1\overline{x}_1$. Since $ 0=\langle A_1^T\overline{y}_1,\overline{x}_1\rangle =\langle \overline{y}_1,A_1\overline{x}_1\rangle$ and $\overline{y}_1$ is positive, we must have $A_1\overline{x}_1=0$. But then, $A_2\overline{x}_2=0$. Since $A_2\leq 0$ and $\overline{x}_2>0$, we must have $A_2=0$. Thus,
$$A=\left [ \begin{array}{cc} A_1 & 0\\0 & A_4\end{array}\right ]$$
with $A_1\overline{x}_1=0$ and $A_4\overline{x}_2=0$. By considering scalar multiples of positive vectors $\overline{x}_1$ and $\overline{x}_2$, we see that $ker(A)$ has dimension at least $2$. This yields a contradiction. Hence, $A$ must be irreducible.
 \end{proof}

\gap

Here is a simple example illustrating the above result.

\begin{example}
In the setting of a matrix game, consider the symmetric irreducible $Z$-matrix
$$A=\left [
\begin{array}{rr}
1 & -1 \\
-1 & 1
\end{array} \right ].
$$
 With ${\bf 1}$ denoting the vector of ones, we see that $A{\bf 1}=0$. So, with $\overline{x}=\overline{y}=\frac{1}{2}{\bf 1}$, we have
$A^T(\overline{y})= 0= A\overline{x}.$
Thus, the value of $A$ is zero. Since the conditions $x\geq 0$, $Ax\geq 0$ impliy that $x$ is a multiple of ${\bf 1}$, we see that $A$ is completely mixed.
\end{example}

\sgap

\noindent{\bf Remarks.} Suppose $A$ is an irreducible $Z$-matrix with value zero. Then, by the above result, there exists $\overline{x}>0$ such that $A\overline{x}=0$. 
Writing  $A=rI-B$, where $B$ is a nonnegative matrix, we have $B\overline{x}=r\overline{x}$. As $B$ is nonnegative and irreducible, by the Perron-Frobenius theorem, $r=\rho(B)$. Thus,  $A=\rho(B)\,I-B$ is now an $M$-matrix. So, {\it every irreducible $Z$-matrix with value zero is an irreducible singular $M$-matrix.} 

\gap

We now present our main theorem.

\begin{theorem}\label{space-irreducibility equals cm}
{\it Let $\V$ be a finite dimensional real inner product space, $K$ be a self-dual cone in $\V$ and $e\in K^\circ$. Suppose $L\in Z(K)$  with $v(L,e)=0$.
Then $(L,e)$ is completely mixed if and only if $L$ is space-irreducible.}
\end{theorem}

\gap

\begin{proof} We are given that $L\in Z(K)$  with $v(L,e)=0$.\\
{\it `If' part:} Assume that $L$ is space-irreducible. 
By   (\ref{schneider-vidyasagar result}), $exp(-tL)\in \pi(K)$ for all $t\geq 0$; also,  by Theorem \ref{E and space-irreducibility}, $exp(-t_0L)$ is $K$-irreducible for some $t_0>0$.  Since  $t_0L\in Z(K)$, $v(t_0L,e)=t_0v(L,e)=0$, and  the game $(L,e)$ is completely mixed if and only if the game $(t_0L,e)$ is completely mixed,  it is enough to show that $(t_0L,e)$ is completely mixed. We assume, without loss of generality, that $t_0=1$. Then,
$L\in Z(K),\,v(L,e)=0$ and $exp(-L)$ is $K$-irreducible. We now show that $(L,e)$ is completely mixed.\\  
Since $L\in Z(K)$,   by Theorem \ref{schneider-vidyasagar},
$$\alpha:=\min\,\{Re(\lambda): \lambda\,\,\mbox{is an eigenvalue of}\,\,L\}$$ 
is an eigenvalue of $L$ with an associated eigenvector $u\in K$. So,
$L(u)=\alpha \,u$ with $0\neq u\geq 0$. Then, $exp(-L)u=exp(-\alpha)\,u$. As $exp(-L)$ is $K$-irreducible, we must have $u>0$ (by Proposition \ref{berman-plemmons result on K-irreducibility}). Thus,
$$L(u)=\alpha\,u\quad\mbox{with}\quad u>0.$$
Since $K$ is self-dual, we see that $L^T$ is a $Z$-transformation on $K$; moreover, $exp(-L^T)$, which is the transpose of $exp(-L)$, is also $K$-irreducible by Proposition \ref{berman-plemmons result on K-irreducibility}.  Since the eigenvalues of $L$ and $L^T$ are the same, we see that $\alpha=\min\,\{Re(\lambda): \lambda\,\,\mbox{is an eigenvalue of}\,\,L^T\}$. So, as in the previous argument, there exists $v>0$ such that $L^T(v)=\alpha\,v$. \\
As $v(L,e)=0$,  let  $(\overline{x},\overline{y})$ be  an (arbitrary) optimal strategy pair so that  
$$L^T(\overline{y})\leq 0\leq L(\overline{x}).$$
We claim  that $\overline{x}>0$.\\
Now, from $L^T(\overline{y})\leq 0$ and $u>0$, we get
$$0\geq \langle L^T(\overline{y}),u\rangle =\langle \overline{y},L(u)\rangle =\alpha \langle \overline{y},u\rangle.$$
Since $0\neq \overline{y}\geq 0$ and $u>0$, we have $\langle \overline{y},u\rangle>0$. Thus, $\alpha\leq 0$.\\
Similarly, by working with $L(\overline{x})\geq 0$ and $L^T(v)=\alpha\,v$, we deduce that $\alpha\geq 0$. Hence, $\alpha=0$. So, 
$L(u)=0=L^T(v)\,\,\mbox{with}\,\,u,v>0$; consequently, $exp(-L)(u)=u$. Now the inequalities $v>0$, $L(\overline{x})\geq 0$,  and $ \langle L(\overline{x}),v\rangle=\langle \overline{x},L^T(v)\rangle=0$ imply that $L(\overline{x})=0$.  From this, we get $exp(-L)(\overline{x})=\overline{x}$. So,  $\overline{x}$ is an eigenvector of $exp(-L)$; as $\overline{x}\in K$, by the $K$-irreducibility of $exp(-L)$, $\overline{x}>0.$ So, we have proved that for any strategy pair $(\overline{x},\overline{y})$, $\overline{x}>0$.
Thus, $(L,e)$ is completely mixed. 
\\
{\it `Only if' part:}
 Now assume that  $(L,e)$ is completely mixed. As $v(L,e)=0$, by Theorem \ref{basic results}, there exists a unique strategy pair $(\overline{x},\overline{y})$ such that $\overline{x},\overline{y}>0$ and 
$$L^T(\overline{y})=0=L(\overline{x}).$$
Additionally, $ker(L)$ (being one-dimensional) consists of multiples of $\overline{x}$. We claim that $L$ is space-irreducible. 
Suppose $F$ is a nonzero face of $K$ such that $L(F-F)\subseteq F-F$. Let $W:=F-F.$ Then, for all $t\geq 0$, $exp(-tL)(W)\subseteq W$. As $L\in Z(K)$, we  have  $exp(-tL)(K)\subseteq K$ for all $t\geq 0$ and so 
$exp(-tL)(W\cap K)\subseteq W\cap K$. However, $W\cap K=F$ as $F$ is a face of $K$. Hence, $exp(-tL)(F)\subseteq F$ for all $t\geq 0$. As $F$ is a proper cone in $W$, by Theorem \ref{schneider-vidyasagar},  $L|_W$ ($L$ restricted to $W$) is a $Z$-transformation on $F$ and 
$$\alpha:=\min\,\{Re(\lambda): \lambda\,\,\mbox{is an eigenvalue of}\,\,L|_W\}$$ 
is an eigenvalue of $L$ with an associated eigenvector $u\in F$. Explicitly, $L(u)=\alpha u$ with $0\neq u\in F$. Then
$$0=\langle L^T(\overline{y}),u\rangle=\langle \overline{y}, L(u)\rangle =\langle \overline{y}, \alpha\,u\rangle=\alpha\,\langle \overline{y},u\rangle.$$
Since $\overline{y}>0$ and $0\neq u\in F\subseteq K$, we have $\langle \overline{y},u\rangle >0$. Hence, $\alpha=0$. At this stage, we have $L(u)=0$, that is, $u\in ker(L)$.
Since $ker(L)$ is one-dimensional and $L(\overline{x})=0$, we see that $u$ is a positive multiple of $\overline{x}$. This proves that $u>0$, that is, $u\in K^\circ$. Since $u$ belongs to the face $F$, we must have $F=K$. This shows that $L$ is space-irreducible. 
\end{proof} 

\gap

\noindent{\bf Remarks.} In an earlier version of the paper, see \cite{gowda-cm-arxiv}, the equivalence of completely mixed property and space-irreducibility was proved under certain restrictions. The implication that the completely mixed property implies space-irreducibility was proved under the additional assumption that each face of $K$ is self-dual in its span, see \cite{gowda-cm-arxiv}, Theorem 3.4. Also, the implication that space-irreducibility implies completely mixed property was proved only for certain types of $Z$-transformations over Euclidean Jordan algebras, see \cite{gowda-cm-arxiv}, Theorem 4.2. The above theorem/proof removes these restrictions.

\gap

In a recent work, Encinas, Mondal, and Sivakumar \cite{encinas et al} study the concepts of range monotone and trivially range monotone properties of Lyapunov and Stein transformations on the space $\Sn$ of all real $n\times n$ symmetric matrices. Their motivation comes from the matrix theory result that every singular irreducible $M$-matrix is almost monotone and has a group inverse, see \cite{berman-plemmons}, Theorem 4.1 and  \cite{encinas et al}, Theorem 1.3. We make three observations: First, a singular $M$-matrix is a $Z$-matrix with  value zero; second, 
 Theorem 4.16 in \cite{berman-plemmons} appears to be related to Kaplansky's
characterization of completely mixed property; and third, Lyapunov and Stein transformations on $\Sn$ are particular examples of $Z$-transformations.  Motivated by these considerations,
we now formulate the following broader result.

\begin{corollary}\label{generalization of encinas et al result} 
 {\it Let $\V$ be a finite dimensional real inner product space, $K$ be a self-dual cone in $\V$ and $e\in K^\circ$. Suppose $L\in Z(K)$, $v(L,e)=0$, and space-irreducible. Then, the following statements hold:
\begin{itemize}
\item [$(a)$] $L$ is almost monotone, that is, $L(x)\geq 0\Rightarrow L(x)=0.$
\item [$(b)$] $L$ is trivially range monotone, that is, $L^2(x)\geq 0\Rightarrow L(x)=0$.
\item [$(c)$] $ker(L^2)=Ker(L)$.
\item [$(d)$] the group inverse of $L$ exists. 
\end{itemize}
}
\end{corollary}

\begin{proof}
Under the specified assumptions, by the above theorem, space-irreducibility implies the completely mixed property.  Now, from Theorem \ref{basic results}, there exist $\overline{x},\overline{y}>0$ such that 
$$L^T(\overline{y})=0=L(\overline{x}).$$
To see $(a)$, suppose $L(x)\geq 0$. Then,
$$0\leq \langle  L(x),\overline{y}\rangle =\langle x,L^T(\overline{y})\rangle =0.$$
As $\overline{y}>0$ and $L(x)\geq 0$, we must have $L(x)=0$.\\
$(b)$: Now suppose $L^2(x)\geq 0$, that is, $L(L(x))\geq 0$ for some $x\in \V$. Then, from $(a)$, $L(L(x))=0$, that is, $L(x)\in Ker(L)$.
We know from Theorem \ref{basic results}, Item (6) that $Ker(L)=\{\lambda \overline{x}:\lambda\in \R\}.$ So,
$L(x)=\lambda \overline{x}$  for some $\lambda\in \R$. We claim that $\lambda=0.$ Suppose $\lambda>0$ so that
$L(x)=\lambda \overline{x}>0$. But then, for any small $\varepsilon >0$, $\overline{x}+\varepsilon x>0$ and $L(\overline{x}+\varepsilon x)=\varepsilon L(x)>0$.
As $L$ is a $Z$-transformation, this implies that $L$ is invertible, see Proposition \ref{positive stability of z}. We reach a contradiction as $L(\overline{x})=0$ with $\overline{x}>0$. Hence, $\lambda$ cannot be positive.
If $\lambda<0$, we can work with $-x$ in place of $x$ to get a contradiction. Hence, $\lambda=0$, so $L(x)=0$. Thus,
$$L^2(x)\geq 0\Rightarrow L(x)=0.$$
This proves $(b)$.\\
When $L^2(x)=0$, we have, from $(b)$, $L(x)=0$. Thus, $Ker(L^2)\subseteq Ker (L)$. As $Ker(L)\subseteq Ker(L^2)$ always, we get the equality in $(c)$.\\
Because of $(c)$, the group inverse of $L$ exists, see \cite{robert}, Theorem 5.
\end{proof}

\begin{example}\label{singular irreducible M-transformation}
Let $\V$ be a finite dimensional real inner product space, $K$ be a self-dual cone in $\V$ and $e\in K^\circ$. On  $\V$, consider the transformation $L=\rho(S)\,I-S$, where $S\in \pi(K)$ with $\rho(S)$ denoting the spectral radius of $S$. 
Suppose $L$ is space-irreducible (equivalently, $S$ is $K$-irreducible, by Proposition \ref{simple observations}). Then, by the Krein-Rutman Theorem (see Theorem 3.2 in \cite{berman-plemmons}) and Proposition \ref{berman-plemmons result on K-irreducibility}, there exists $\overline{x},\overline{y}>0$ such that $S^{T}(\overline{y})=\rho(S)\,\overline{y}$ and $S(\overline{x})=\rho(S)\,\overline{x}$. We see that  $L^{T}(\overline{y})=0=L(\overline{x})$ and so $v(L,e)=0$. Moreover, by Theorem \ref{space-irreducibility equals cm}, $(L,e)$ is completely mixed. In analogy with matrices, we may call $L$, a `singular $M$-transformation'. So, as in the matrix case, {\it a singular space-irreducible $M$-transformation is completely mixed.}
A simple example of such an $L$ can be seen by defining  $S(x)=\langle x,e\rangle e$. To get more general $L$, we may consider  
 a `strictly positive transformation' which satisfies the condition $0\neq x\geq 0\Rightarrow S(x)>0$. Then, $S$ is $K$-irreducible (by Proposition \ref{berman-plemmons result on K-irreducibility}) and  $\rho(S)>0$.  Thus the game $(L,e)$, where $L=\rho(S)\,I-S$,  has value zero and is completely mixed. In this way, we get a large collection of $Z$-transformations with value zero and having the completely mixed property. 
\end{example}

\section{Completely mixed and irreducibility results on symmetric cones in Euclidean Jordan algebras}
In this section, we specialize by assuming that the given proper cone $K$ is a symmetric cone which, by definition, is self-dual and homogeneous. As every symmetric cone appears as the cone of squares in a Euclidean Jordan algebra, we use the Jordan algebraic machinery to describe our completely mixed property and irreducibility results and also to provide interesting examples. First, we cover some background material. \\

For basic definitions and results on Euclidean Jordan algebras (particularly, those that are not explicitly mentioned/referenced), we refer to \cite{faraut-koranyi,gowda-sznajder-tao}.
Let $(\V, \circ, \langle\cdot,\cdot\rangle)$ denote a  Euclidean Jordan algebra  with unit element $e$; 
here, for any two elements $x,y$, the Jordan product and inner product are denoted, respectively,  by $x\circ y$ and $\langle x,y\rangle$. 
It is well-known that 
any Euclidean Jordan algebra is a direct product/sum
of simple Euclidean Jordan algebras and every simple Euclidean Jordan algebra is isomorphic to one of five algebras,
three of which are the algebras of $n\times n$ real/complex/quaternion Hermitian matrices. The other two are: the algebra of $3\times 3$ octonion Hermitian matrices and the Jordan spin algebra $\Ln$. In the algebras $\Hn$ (of all $n\times n$ complex Hermitian matrices) and $\Sn$ (of all $n\times n$ real symmetric matrices), 
 the Jordan product and the inner product are given by 
$X\circ Y:=\frac{XY+YX}{2}\quad \mbox{and}\quad \langle X,Y\rangle:=tr(XY).$ The algebra $\Rn$ carries the componentwise product (as the Jordan product) and the usual inner product.
 
In the Euclidean Jordan algebra $\V$, the set 
$$K=\V_+:=\{x\circ x:x\in \V\}$$
is (called) the symmetric cone of $\V$. It is a {\it self-dual cone}, so 
$x\in K\Leftrightarrow \langle x,y\rangle \geq 0$ for all $y\in K.$ 
Morover, $u\in K^\circ \Leftrightarrow \langle u,y\rangle>0$ for all $0\neq y\in K$.  In the case of $\Hn$ (or $\Sn$), $K$ is the cone of positive semidefinite matrices.\\
In $\V$, the following holds \cite{gowda-sznajder-tao}: When $x,y\in K$,
\begin{equation}\label{zero jordan product}
x\circ y=0\Leftrightarrow \langle x, y\rangle=0.
\end{equation}

\gap

An element $c$ in $\V$ is an {\it idempotent} if $c^2=c$; it is a {\it primitive idempotent} if it is nonzero and cannot be written as sum of two other nonzero idempotents. A {\it Jordan frame}
 $\{e_1,e_2,\ldots, e_n\}$ in $\V$ consists of
primitive idempotents that are mutually orthogonal and with sum equal to the unit element $e$.  It is known that all Jordan frames in $\V$ have the same number of elements -- called the 
rank of $\V$. 
{\it With $n$ denoting the rank of $\V$}, we have the 
 {\it spectral decomposition theorem} (\cite{faraut-koranyi}, Theorem III.1.2):
{\it Every $x$ in $\V$ can be written as 
 $$x=x_1e_1+x_2e_2+\cdots+x_ne_n,$$
where the real numbers $x_1,x_2,\ldots, x_n$ are (called) the eigenvalues of $x$ and
$\{e_1,e_2,\ldots, e_n\}$ is a Jordan frame in $\V$. 
}
Then, we define the rank of $x$ -- denoted by $\mbox{rank}\, x$ -- as the number of nonzero eigenvalues of $x$.
Corresponding to the above decomposition, we define the trace of $x$ as  
$\tr(x):=x_1+x_2+\cdots+x_n.$
It is known that $(x,y)\mapsto \tr(x\circ y)$ defines another inner product on $\V$ that is compatible with the Jordan product.
{\it Throughout this paper we assume that the inner product on $\V$
is this trace inner product, that is,
$\langle x,y\rangle=\tr(x\circ y).$} 

\gap

Given an idempotent $c$ and $\gamma\in \{1,\frac{1}{2},0\}$, we let 
$$\V(c,\gamma):=\{x\in \V: x\circ c=\gamma\,x\}.$$ 
It is known, see \cite{faraut-koranyi}, Prop. IV.1.1, that  $\V(c,1)$ and $\V(c,0)$ are subalgebras of $\V$ and
$$\V(c,1)\circ \V(c,0)=\{0\}.$$ More importantly,  the following {\it Peirce (orthogonal) decomposition} holds:
\begin{equation}\label{peirce decomposition}
\V=\V(c,1)+\V(c,\frac{1}{2})+\V(c,0).
\end{equation}

Let $\V_+(c,1)$ denote the symmetric cone of the (sub)algebra $\V(c,1)$. Then, 
$$\V(c,1)=\V_+(c,1)-\V_+(c,1)$$ and, moreover, $\V_+(c,1)$ is self-dual in $\V(c,1)$. Regarding the faces of a symmetric cone, we have the following:

\begin{proposition}  (\cite{gowda-sznajder}, Theorem 3.1)\label {face of a symmetric cone}
{\it Every face of  the symmetric cone $\V_+$ is of the form $\V_+(c,1)$ for some idempotent $c$. }
\end{proposition}

\sgap

If $\{e_1,e_2,\ldots, e_n\}$ is a Jordan frame in $\V$, then
we have (another) {\it Peirce orthogonal decomposition} 
(Theorem IV.2.1 in \cite{faraut-koranyi}):
\begin{equation}\label{POD}
\V=\sum_{1\leq i\leq j\leq n}\V_{ij},
\end{equation}
where $\V_{ii}:=\V(e_i,1)=\R\, e_i$ for all $i$ and for $i\neq j$, 
$\V_{ij}:=\V(e_i,\frac{1}{2})\cap \V(e_j,\frac{1}{2}).$
\\

Given $a\in \V$, the corresponding {\it Lyapunov transformation} and {\it quadratic representation} are defined by:
$$L_a(x):=a\circ x\quad\mbox{and}\quad P_a(x):=2a\circ (a\circ x)-a^2\circ x\quad (x\in \V).$$
These two transformations are self-adjoint, $L_a\in Z(K)$, and $P_a\in \pi(K)$.

\sgap

Given (\ref{POD}), consider $a,x\in \V$ with
$a=a_1e_1+a_2e_2+\cdots+a_ne_n$ and
$$x=\sum_{1\leq i\leq j\leq n}x_{ij}\quad (x_{ij}\in \V_{ij}).$$
 Then we have the well-known formulae
\begin{equation}\label{formula for La}
L_a(x)=\sum_{1\leq i\leq j\leq n}\Big (\frac{a_i+a_j}{2}\Big )x_{ij} \quad\mbox{and}\quad P_a(x)=\sum_{1\leq i\leq j\leq n}a_ia_j\,x_{ij}.
\end{equation}

Let $L$ be a Lyapunov-like transformation on the symmetric cone $K$. Then, its symmetric part $\frac{L+L^T}{2}$ and skew-symmetric part $\frac{L-L^T}{2}$ are both Lyapunov-like. Moreover, it is known, see \cite{tao-gowda-representation}, that the symmetric part is of the form $L_a$ for some $a\in \V$ and the skew-symmetric part is a derivation $D$ (which means that  $D(x\circ y)=Dx\circ y+x\circ Dy$ for all $x,y\in \V$). So, in the setting of a Euclidean Jordan algebra,
\begin{equation}\label{LL representation in EJA}
L\in \LL(K)\Leftrightarrow L=L_a+D\,\,\mbox{for some}\,\,a\in \V\,\mbox{and derivation}\,\,D.
\end{equation}

It is known (a result due to Damm, see \cite{tao-gowda-representation}) that in the Euclidean Jordan algebra $\Hn$, every Lyapunov-like transformation is of the form $L_A$ for some complex matrix $A$, where
$$L_A(X)=AX+XA^*\quad(X\in \Hn).$$
A similar statement holds in $\Sn$  with $A$ real and $A^*$ replaced by $A^T$. 

\gap

{\it Henceforth, in all game-theoretic results/concepts dealing with Euclidean Jordan algebras, we  assume that $e$ is the unit element and write $v(L)$ in place of $v(L,e)$, etc.}
\subsection{A generalization of a result of Parthasarathy et al.}

The following is a generalization of the recent result by Parthasarathy et al. \cite{parthasarathy et al},  mentioned in the Introduction. 
\begin{theorem}\label{generalization of tp et al result}
 {\it Suppose $L$ is a linear transformation on the  Euclidean Jordan algebra $\V$ with $v(L)=0$.
Then, $L$ is completely mixed if and only if $v(L+P_c)>0$  for every primitive idempotent $c$ in $V$.
}
\end{theorem}

\begin{proof} It is given that $v(L)=0$. Take any primitive idempotent $c$ in $V$. As $P_c(K)\subseteq K$,
$\langle (L+P_c)x,y\rangle \geq \langle L(x),y\rangle$ for all strategies $x$ and $y$; hence, by the min-max theorem mentioned in the Introduction, 
$v(L+P_c)\geq v(L)=0.$

Now suppose $L$ is completely mixed. Then there exist optimal strategies $\overline{x}, \overline{y}>0$ such that $L^{T}(\overline{y})=0=L(\overline{x})$. Additionally, 
 for (arbitrary) strategies $z,w$,
$$L^{T}(z)\leq 0\leq L(w)\Rightarrow (z,w) \,\mbox{optimal}\Rightarrow z,w>0\Rightarrow  L^{T}(z)= 0= L(w).$$
Now, assume if possible, $v(L+P_c)=0$ for some $c$. Then, there is a  strategy $z$ such that 
$$(L+P_c)^T(z)\leq 0.$$ Since $P_c$ is self-adjoint and $P_c(K)\subseteq K$,  we have $L^T(z)\leq 0.$ As $ L(\overline{x})=0$, we now have 
$$L^{T}(z)\leq 0=L(\overline{x}).$$
 From the above implications, $z>0$ and $L^T(z)=0$. But then,
$P_c(z)\leq 0$ with $z>0$. Again, since $P_c(K)\subseteq K$, we must have $P_c(z)=0$. As $z>0$ and $0\neq c\in K$, we must have $\langle z,c\rangle >0$. However, $0<\langle z,c\rangle =\langle z, P_c(c)\rangle =\langle P_c(z),c\rangle =\langle 0,c\rangle =0.$ We reach a contradiction. Thus, the value of $L+P_c$ must be  positive.\\
To see the reverse implication, suppose the value of $L+P_c$ is positive for all primitive idempotents $c$. Assume that $L$ is not completely mixed. Then, there  are strategies $x$ and  $y$
such that $L^T(y)\leq 0\leq L(x)$, where some eigenvalue of $y$, say,  $y_1$ is zero. 
we write the spectral decomposition $y=y_1e_1+y_2e_2+\cdots+y_ne_n=0e_1+y_2e_2+\cdots+y_ne_n$ and let $c=e_1$ (which is a primitive idempotent).
Then, by (\ref{formula for La}), $P_c(y)=0$. As $P_c(x)\geq 0$, we have $(L+P_c)^{T}(y)\leq 0\leq (L+P_c)(x)$. This implies that the value of $L+P_c$ is zero, leading to a contradiction. Thus we have the reverse implication.
\end{proof}

\subsection{Irreducibility of a linear transformation on a symmetric cone}
Now consider a Euclidean Jordan algebra with its symmetric cone $K$ ($=\V_+$) and unit element $e$. In this setting, as noted in Proposition \ref{face of a symmetric cone}, every face of $K$ is of the form $\V_+(c,1)$ (the symmetric cone of $\V(c,1)$) for some idempotent $c$. Moreover,  $\V_+(c,1)=\{0\}$ or $\V_+$ if and only if $c=0$ or $e$. Since $\V(c,1)=\V_+(c,1)-\V_+(c,1)$, Definition \ref{irreducibility defn over a proper cone} reduces to the  following.

\begin{definition}\label{irreducibility defn over eja}
{\it Let $L$ be a linear transformation on a Euclidean Jordan algebra $\V$. Then, $L$ is  space-irreducible if 
$$L\big(\V(c,1)\big)\subseteq \V(c,1)\Rightarrow c=0\,\,\mbox{or}\,\,e.$$
 }
\end{definition}

So, space-irreducibility means that other than $\{0\}$ and $\V$, $L$ does not have any invariant subalgebras of the form $\V(c,1)$. 
We note that when $L$ is not space-irreducible, with $c\neq 0,e$, $W:=\V(c,1)$, and $L(W)\subseteq W$, we can write $\V=W^\perp+W$ and represent $L$ in the block form 
$$\left [ \begin{array}{cc}
L_1 & L_2\\0& L_3\end{array} \right ],$$
where $L_1:W^\perp\rightarrow W^\perp,\,L_2:W\rightarrow W^\perp$, and $L_3:W\rightarrow W$ are linear transformations.

\gap

The following two examples deal with the space-irreducibility of $L_A$ on $\Sn$ and $\Hn$.

\begin{example}\label{irreducibility of L_A on sn}
Let $n>1$. For any $A\in \R^{n\times n}$, consider the transformation $L_A$ on $\Sn$ defined by $L_A(X)=AX+XA^T$.
{\it  We claim that  $L_A$ is space-irreducible if and only if for every orthogonal matrix $U$, the matrix $U^TAU$ is irreducible.} We sketch a proof. Suppose $L_A$ is space-reducible, that is, there exists an idempotent 
$C\,(\neq 0, I)$ in $\Sn$ such that $L_A(\V(C,1))\subseteq \V(C,1)$. As $C$ is an idempotent, its eigenvalues are either $0$ or $1$; we can write $C=UPU^T$, where $U$ is orthogonal and $P$ is a diagonal matrix written in the block form as
$$P=\left [
\begin{array}{cc}
I_k & 0\\
0 & 0
\end{array}\right],$$
with $I_k$ denoting the identity matrix of size $k\times k$, $1\leq k<n$.  
Since $X\in \V(C,1)$ if and only if $C\circ X=X$, that is, $\frac{CX+XC}{2}=X$, we see that $\V(C,1)=U\V(P,1)U^T$. Moreover, the condition 
$L_A(\V(C,1))\subseteq \V(C,1)$ is equivalent to $L_B(\V(P,1))\subseteq \V(P,1)$, where $B=U^TAU$. Now, every element in $\V(P,1)$ has the same block form as $P$, except that in place of $I_k$, we have 
 a matrix from ${\cal S}^k$. So then $L_B(P)\in \V(P,1)$ implies that the southwest block of the matrix $B$ is zero, implying that 
$B$ is reducible. Now suppose there is an orthogonal $U$ such that the matrix $B=U^TAU$ is reducible.  By choosing an appropriate diagonal matrix $P$ and reversing the argument given above, one can construct an idempotent matrix $C\,(\neq 0,I)$ for which $L_A(\V(C,1))\subseteq \V(C,1)$ holds. Thus, we have shown that $L_A$ is space-reducible if and only if for some orthogonal matrix $U$, $U^TAU$ is reducible. This proves our claim. \\
Now, by the real version of Schur's upper triangularization theorem (\cite{horn-johnson}, Theorem 2.3.4), $A$ is orthogonally similar to a real block upper triangular matrix with each diagonal block either a $1\times 1$-matrix  or a $2\times 2$-matrix with a non-real pair of complex conjugate eigenvalues. So, if $L_A$ is space-irreducible (in which case, by our claim above, every matrix of the form $UAU^T$ is irreducible),  $A$ must be a $2\times 2$ matrix with a non-real pair of complex conjugate eigenvalues. Conversely, if $A$ is a $2\times 2$ matrix with this specific property, then every $U^TAU$ is irreducible; hence, $L_A$ is space-irreducible. In summary, we have proved:\\
{\it On $\Sn\,(n>1)$, $L_A$ is space-irreducible if and only if $n=2$ and $A$ is a real matrix with a non-real pair of complex conjugate eigenvalues.}\\
 To see a specific example,  consider the Lyapunov transformation $L_A$ on ${\cal S}^2$ defined by
$$L_A(X)=AX+XA^T\quad\mbox{where}\quad A=\left [ \begin{array}{rr} 0 &1\\ -1 &0\end{array}\right ].$$
As $A+A^T=0$, we see that  $L_A+(L_A)^T=0$. We also have $L_A(I)=0=(L_A)^T(I)$ and so, $v(L_A,I)=0$. Additionally,  if $X$ in ${\cal S}^2$ is positive semidefinite with $L_A(X)\geq 0$, then $X$ is a multiple of $I$. This shows that $L_A$ is completely mixed, hence irreducible.   It is easy to verify that this $L_A$ cannot be of the form $rI-S$ for any $S\in \pi({\cal S}^2_+).$ {\it Thus,  in contrast to the matrix game case (cf. Remarks made after Theorem \ref{z-matrix result}), we have an example of a Lyapunov-like transformation with value zero that is completely mixed and which is not a singular  $M$-transformation. }  
\end{example}

\begin{example}
For any $n\times n$ complex matrix $A$, $n>1$, consider $L_A$ on $\Hn$ defined by $L_A(X)=AX+XA^*$, where $A^*$ is the adjoint/conjugate of $A$. By Schur's upper-triangularization theorem (\cite{horn-johnson}, Theorem 2.3.1), we can write $A$ as
$UQU^*$, where $Q$ is an upper-triangular matrix and $U$ is unitary.  We note that the matrix $Q$ is reducible.  We can mimic the argument given in the previous example to show that $L_A$ keeps some nontrivial $\V(C,1)$
invariant. This means that {\it over $\Hn$ $(n>1)$, every $L_A$ is space-reducible.} By Theorem \ref{space-irreducibility equals cm}, no such transformation can (simultaneously) have value zero and be completely mixed.
Thus, because of Theorem \ref{GR- nonzero value implies cm}, {\it Over $\Hn$ $(n>1)$, $L_A$ is completely mixed if and only if its value is nonzero.}
\end{example}

\subsection{Irreducibility of a positive transformation on a symmetric cone}
The following theorem is a modified form of Proposition \ref{berman-plemmons result on K-irreducibility} stated for Euclidean Jordan algebras. We emphasize that the natural number $n$ appearing in Item $(iv)$ below is the rank (and not the dimension) of the algebra.
This result is motivated by a similar result in the setting of Hermitian matrices, see \cite{wolf}, Theorem 6.2 or \cite{idel}, Proposition C.1.

First, we state a lemma which is motivated by the result that when a (Hermitian) positive semidefinite matrix is written in a block form, if a diagonal block is zero, then the corresponding off-diagonal blocks are also zero.

\begin{lemma}\label{zero mixed term}
{\it For an idempotent $c$, consider the Peirce  orthogonal decomposition (\ref{peirce decomposition}). Suppose an element $z\in \V$ has no component in $\V(c,0)$, that is,  
$$ z=x+y,$$
where $x\in \V(c,1)$ and $y\in \V(c,\frac{1}{2})$. If $z\geq 0$, then, $y=0.$
}
\end{lemma}

\begin{proof}
 Since $e-c\in \V(c,0)$, by the orthogonality of the spaces involved, we have $\langle x+y,e-c\rangle =0$. As $z=x+y\geq 0$ and $e-c\geq 0$, by (\ref{zero jordan product}), we have
 $$(x+y)\circ (e-c)=0.$$
This leads to $(x+y)\circ e-(x+y)\circ c=0$. Now, $x\in \V(c,1)\Rightarrow x\circ c=x$ and $y\in \V(c,\frac{1}{2})\Rightarrow y\circ c=\frac{1}{2} y$. Hence,  
$$(x+y)-(x+\frac{1}{2}y)=0.$$
This simplifies to  $y=0$.
\end{proof}

\gap

\begin{theorem}\label{K-irreducibility in eja}
{\it Suppose $\V$ is a Euclidean Jordan algebra  with unit $e$, $K=\V_+$, and $T\in \pi(K)$. Let $c$ denote an idempotent. Then the following are equivalent:
\begin{itemize}
\item [$(i)$]  $\langle T(c),e-c\rangle=0\Rightarrow c=0\,\, \mbox{or}~~ e$.
\item [$(ii)$] $T$ is space-irreducible, that is,  $T\big (\V(c,1)\big)\subseteq \V(c,1)\Rightarrow c=0\,\mbox{or}\,e$.
\item [$(iii)$] $T$ is $K$-irreducible.
\item [$(iv)$] For each $0\neq x\geq 0$, it holds that $(I+T)^{n-1}x>0$, where $n$ denotes the rank of $\V$.
\item [$(v)$] For each $0\neq x\geq 0$ and every $t>0$, $exp(tT)x>0$.

\end{itemize}
}
\end{theorem}

\begin{proof}
$(i)\Leftrightarrow (ii)$:  Suppose the negation of $(i)$ holds so that there is an idempotent $c\,(\neq  0,e)$ such that 
$\langle T(c),e-c\rangle=0$. Let $d:=e-c$ so that $d\neq 0,e$ and $\langle T(c),d\rangle =0$. Then, in the algebra $\V(c,1)$,  $c$ is the unit element and in $\V(c,0)$ (which is $\V(d,1)$), $d$ is the unit element. Now take primitive idempotents $c_1$
 in $\V(c,1)$ and $d_1$ in $\V(c,0)$. Since $0\leq c_1\leq c$ and $0\leq d_1\leq d$, we have, by the positivity of $T$,
 $0\leq \langle T(c_1),d_1\rangle\leq  \langle T(c),d\rangle =0$. Thus, $\langle T(c_1),d_1\rangle =0$. Since (from the spectral theorem) any element in $\V(c,1)$ is a linear combination of primitive idempotents in $\V(c,1)$ and any element in $\V(d,1)$ is a linear combination of primitive idempotents in $\V(d,1)$, we see that 
$$ T\big (\V(c,1)\big )\perp \V(c,0).$$
Now take any element $x\in \V_+(c,1)$ and consider the Peirce orthogonal decomposition of $T(x)$:
$$T(x)=u+v+w,\quad\mbox{where}\quad u\in \V(c,1), v\in \V(c,\frac{1}{2}), w\in \V(c,0).$$
Since $T(x)\perp w$, we must have $w=0$. So now, $0\leq T(x)=u+v$. By Lemma \ref{zero mixed term}, we have $v=0$. Thus, for any $x\in \V_+(c,1)$, $T(x)\in \V(c,1)$. Since any element in $\V(c,1)$ is a difference of two elements in $\V_+(c,1)$, we see that 
$T\big (\V(c,1)\big)\subseteq \V(c,1)$. Hence, the negation of $(ii)$ holds. 
\\
Now suppose the negation of $(ii)$ holds so that there exist $c \,(\neq 0,e)$ with $T(c)\in T\big (\V(c,1)\big)\subseteq \V(c,1)$.  Then, with $0\neq d:=e-c$, we have $ \V(c,1)\perp \V(d,1)$, hence $\langle T(c),d\rangle=0$. This gives the negation of $(i)$. 
Hence, $(i)$ and $(ii)$ are equivalent.\\
$(ii)\Leftrightarrow (iii)$: This follows from the facts that for $T\in \pi(K)$, $T(\V(c,1))\subseteq \V(c,1)$ if and only if $T(\V_+(c,1))\subseteq \V_+(c,1),$ and that every face of (the symmetric cone) $K$ is of the form of  $\V_+(c,1)$ for some idempotent $c$.\\
$(ii)\Rightarrow (iv)$: Assume that $(ii)$ holds. 
Let $0\neq x\geq 0$. If $x>0$, $(iv)$ trivially holds as $T\in \pi(K)$. So, suppose $rank\,(x)=k$, where $1\leq k<n$.  
We write the spectral decomposition $x=x_1e_1+x_2e_2+\cdots+x_ke_k+0e_{k+1}+\cdots+0e_n$, where $\{e_1,e_2,\ldots, e_n\}$ is a Jordan frame and $x_i>0$ for all $i=1,2,\ldots, k$. Let 
$c:=e_1+e_2+\cdots+e_k$. Then, in the (sub)algebra $\V(c,1)$, $x>0$. We show that $T(x)\not\in \V(c,1)$. Assume the contrary, i.e., $T(x)\in \V(c,1)$; then, as $T\in \pi(K)$, $T(x)\in \V_+(c,1)$.
{\it We now claim that  $T(\V(c,1))\subseteq \V(c,1)$}.
To see this,  take any $y\in \V(c,1)$. Then,
for small positive $\varepsilon$, $x\pm \varepsilon y>0$ in $\V(c,1)$. As $T\in \pi(K)$ and $\V_+(c,1)\subseteq K$, we have $T(x\pm \varepsilon y)\geq 0$. Since $T(x)$ (which is in $\V_+(c,1)$) is now 
the average of two elements (namely, $T(x+\varepsilon y)$ and $T(x-\varepsilon y)$) in $K$, these two elements belong to the face $\V_+(c,1)$; consequently, their difference $2\varepsilon T(y)$ (equivalently, $T(y)$) belongs to $\V(c,1)$. This proves the claim. However, this contradicts $(ii)$ and so we must have  $T(x)\not \in \V(c,1)$. Since $x\in \V(c,1)$, this implies that 
$$z:=x+T(x)=(I+T)x\not\in \V(c,1).$$   
We now claim that 
$$rank\,z>\,rank\,x=k.$$ 
Suppose, if possible, $l:=rank\,z\leq k$. Then, we write the spectral decomposition of $z$ as $z=z_1c_1+z_2c_2+\cdots+z_lc_l+0c_{l+1}+\cdots+0c_n$, where $\{c_1,c_2,\ldots, c_n\}$ is a Jordan frame and $z_i\neq 0$ for $1\leq i\leq l$.  Noting that $l\leq k$, we consider the idempotent $d:=c_{k+1}+c_{k+2}+\cdots+c_n$. Clearly,  $\langle z,d\rangle =0$. This implies (recall $x,T(x)\geq 0$)  that $\langle x,d\rangle=0=\langle T(x),d\rangle.$ From $0=\langle x,d\rangle=\sum_{i=1}^{k}x_i\langle e_i,d\rangle$ with $x_i>0$ for all $i$ appearing in the sum, we see that $\langle e_i,d\rangle =0$ for $i=1,2\ldots, k$. It follows that $\{e_1,e_2,\ldots,e_k\}\perp \{c_{k+1},c_{k+2},\ldots,c_n\}$.  As the rank of $\V$ is $n$, 
$\{e_1,e_2,\ldots,e_k, c_{k+1},c_{k+2},\ldots,c_n\}$ forms a Jordan frame; in particular, $d=e-c$. From $\langle T(x),d\rangle =0$ and (\ref{zero jordan product}),  we have, 
 $T(x)\circ d=0$.  Consequently, as $d=e-c$, we have $$ T(x)\in \V(d,0)=\V(c,1).$$
We reach a contradiction to the earlier assertion that $T(x)\not \in \V(c,1)$. This proves  that 
$$0\neq x\in K,\,\mbox{rank}\,x<n\Rightarrow \mbox{rank}\,(I+T)x>\mbox{rank}\,x.$$
So, starting with an element $0\neq x\in K$ and applying $I+T$ repeatedly,  we reach an element in  $K$ with rank $n$ in at most $n-1$ steps. Thus,  $(iv)$ holds.\\
$(iv)\Rightarrow (v)$: Suppose $(iv)$ holds, but not $(v)$. Then, there exist  $0\neq x\geq 0$  and a positive number $t$ such that $exp(tT)x\not >0$. (We note that $exp(tT)\in \pi(K)$.) 
Then, for some (primitive) idempotent $c$, $\langle exp(tT)x,c\rangle =0$. Expanding $exp(tT)$ and using the positivity of $T$ and its powers, we see that $\langle T^m(x),c\rangle =0$ for all $m=0,1,2,\ldots$ Hence, $\langle (I+T)^{n-1}x,c\rangle =0$. This violates $(iv)$. \\
$(v)\Rightarrow (ii)$: Suppose $(v)$ holds but not $(ii)$. Then, 
for some idempotent $c\,(\neq 0,e)$, we have $T\big (\V(c,1)\big )\subseteq \V(c,1)$. As $c\in \V(c,1)$, we have $T^m(c)\in \V(c,1)$ for all $m=0,1,2,\ldots$ Thus, 
$exp(T)c\in \V(c,1)$. In particular, $exp(T)c\not >0$, contradicting $(v)$ for $x=c$ and $t=1$.
\\
In summary, we have proved the implications $(i)\Leftrightarrow (ii)\Leftrightarrow (iii)\Rightarrow (iv)\Rightarrow (v)\Rightarrow (ii)$. Thus, all the statements in the theorem are equivalent.
This completes the proof.
\end{proof}

\subsection{The completely mixed property of symmetric/skew-symmetric Lyapunov-like transformations}
In the next two results, we deal with the completely mixed property of a symmetric/skew-symmetric Lyapunov-like transformation. Recall that 
every symmetric Lyapunov-like transformation is of the form $L_a$ for some $a\in \V$ and a skew-symmetric Lyapunov-like transformation is a derivation, see \cite{tao-gowda-representation}.

\begin{theorem} 
{\it Suppose $a\in \V$. Then the following statements are equivalent:
\begin{itemize}
\item [$(i)$] $L_a$ is completely mixed.
\item [$(ii)$] $a>0$ or $a<0$.
\item [$(iii)$] $v(L_a)\neq 0$.
\end{itemize}
}
\end{theorem}

\begin{proof} Without loss of generality, we assume that the dimension of $\V$ is more than one.\\
$(i)\Rightarrow (ii)$: Suppose $L_a$  is completely mixed. Then, there exists $\overline{x}>0$ with $L_a(\overline{x})=0$ and $ker(L_a)=\{t\overline{x}:\,t\in \R\}.$ Consider  the spectral decomposition $a=a_1e_1+a_2e_2+\cdots+a_ne_n$. If possible, let $a_1=0$. Then, by (\ref{formula for La}),
$$L_a^T(e_1)=L_a(e_1)=0.$$
As  $e_1\not >0$, we reach a contradiction to $(i)$. Hence, $a_i\neq 0$ for all $i$. Now suppose, if possible, there is an index $k$ such that $1\leq k<n$ such that
$a_i>0$ for $1\leq i\leq k$ and $a_i<0$ for $k+1\leq i\leq n$. Let $c:=e_1+e_2+\cdots+e_k$. Then, by (\ref{formula for La}),  $L_a(c)=\sum_{i=1}^{k}a_ie_i\geq 0$ and 
$L_a^{T}(e-c)=\sum_{i=k+1}^{n}a_ie_i\leq 0.$ We see that
$$L_a^{T}\Big (\frac{e-c}{n-k}\Big )\leq 0\leq L_a\Big (\frac{c}{k}\Big ).$$
As $\frac{c}{k}\not >0$, we reach  a contradiction to $(i)$. Hence, $a>0$ or $a<0$. 
\\
$(ii)\Rightarrow (iii)$: We  show that $a>0\Rightarrow v(L_a)>0.$ Similarly $a<0\Rightarrow v(L_a)<0$. \\
Given $a>0$, suppose $v(L_a)\leq 0$. Then, there exists $y\geq 0$ with $\langle y,e\rangle=1$ (in particular, $y\neq 0$) such that $L_a^T(y)\leq v(L_a)\,e\leq 0$. As $a>0$, we have 
$\langle L_a^{T}(y),a\rangle \leq 0$, that is, $\langle y, L_a(a)\rangle \leq 0$. Since $y\geq 0$ and $L_a(a)=a^2>0$, we must have $y=0$, yielding a contradiction.
\\
$(iii)\Rightarrow (i)$: As $L_a$ is Lyapunov-like, this is a special case of Theorem 7 in \cite{gowda-value}.
\end{proof} 

\gap

Suppose the linear transformation $L$ is skew-symmetric, that is, $L+L^T=0$. In this case, the game $(L,e)$ is called a {\it symmetric game} \cite{gokulraj-chandrashekaran-1}. For such an $L$, the value is zero (see
Theorem \ref{basic results}, Item (2)). It is easy to see that a skew-symmetric $Z$-transformation is Lyapunov-like.

\sgap

\begin{theorem}  \label{cm for derivation}
{\it Let $D$ be a skew-symmetric $Z$-transformation (=derivation) on a Euclidean Jordan algebra $\V$. Then, the following are equivalent:
\begin{itemize}
\item [(i)] The game $(D,e)$ is completely mixed.
\item [(ii)] $D$  is space-irreducible.
\item [(iii)] $ker(D)$ is one dimensional. 
\end{itemize}
}\end{theorem}

\sgap

\begin{proof}
Since  $D$ is skew-symmetric, its value is zero. As $D$ is a $Z$-transformation,  by Theorem \ref{space-irreducibility equals cm}, $(i)$ and $(ii)$ are equivalent.\\ 
$(i)\Rightarrow (iii)$: This follows from Theorem \ref{basic results}, Item (6).\\
$(iii)\Rightarrow (i)$: As $D$ is a derivation, $D(e)=0$, see \cite{tao-gowda-representation}, Proposition 3. Furthermore,  as $D$ is skew-symmetric, $D^T(e)=0$. Now suppose $(\overline{x},\overline{y})$ is any optimal strategy pair for the game $(D,e)$ so that  $D^T(\overline{y})\leq 0\leq D(\overline{x})$. Since $(e,e)$ is also an optimal strategy pair and $e>0$, we have 
 $D^T(\overline{y})= 0= D(\overline{x})$, see Theorem \ref{basic results}, Item (1).
 So, when $(iii)$ holds, $\overline{x}$ is a positive multiple of $e$; thus, $\overline{x}>0$. This proves $(i)$.
\end{proof}
 
We refer to Example \ref{irreducibility of L_A on sn} for an illustration of the above theorem.

\section{Concluding remarks}
In the setting of a self-dual cone in a finite dimensional real inner product space,  we showed that a $Z$-transformation with value zero is completely mixed if and only if it is  space-irreducible. We also gave a finer characterization of  cone-irreducibility for a positive transformation on 
a Euclidean Jordan algebra and generalized a recent result of Parthasarathy et al. to the setting of a symmetric cone. As mentioned in the Introduction, one can look for the properties of $Z$-transformations with value zero in regard to dynamical systems and complementarity theory. These are left for future study.

\gap

\noindent{\bf Acknowledgments:} We wish to thank   G. Ravindran, Indian Statistical Institute, Chennai, for discussions on Theorem \ref{z-matrix result}, especially for providing reference \cite{li}. Special thanks are due to M. Orlitzky \cite{orlitzky-seminar} for pointing out the equivalence between our space-irreducibility concept and the irreducibility concept of Elsner. That equivalence and the reference   \cite{elsner} enabled us to substantially improve the results presented in \cite{gowda-cm-arxiv}. 

\end{document}